\documentclass[10pt]{amsart}

\usepackage[margin=1in]{geometry} 
\usepackage{amsmath,amssymb,bbm,bm}
\usepackage{amsthm}
\usepackage{graphicx, xcolor}
\usepackage{hyperref}
\usepackage{amssymb,amsfonts}
\usepackage{arydshln}
\usepackage[all,arc]{xy}
\usepackage{enumerate}
\usepackage{mathrsfs}
\usepackage{xcolor}
\usepackage{enumerate, enumitem}

\newcommand{\N}{\mathbb{N}}

\newcommand{\Oc}{\mathcal{O}}

\newcommand{\fp}{\mathrm{fp}}
\newcommand{\FP}{\mathrm{FP}}
\newcommand{\mix}{\mathrm{mix}}
\newcommand{\rel}{\mathrm{rel}}

\DeclareRobustCommand{\stirling}{\genfrac\{\}{0pt}{}}

\renewcommand{\epsilon}{\varepsilon} 

\newcommand{\Address}{{
\bigskip
\footnotesize

\textsc{Department of Mathematics \& Statistics, McMaster University, Hamilton, ON, L8S 4K1, Canada}\par\nopagebreak
\textit{E-mail address}: \texttt{paguyoj@mcmaster.ca}
}}

\def\bal#1\eal{\begin{align*}#1\end{align*}}

\newtheorem{theorem}{Theorem}[section]
\newtheorem{lemma}[theorem]{Lemma}
\newtheorem{proposition}[theorem]{Proposition}
\newtheorem{corollary}[theorem]{Corollary}

\theoremstyle{definition}
\newtheorem{example}{Example}[section]

\title[Mixing times of a Burnside process on set partitions]{Mixing times of a Burnside process Markov chain on set partitions}
\author{J. E. Paguyo}


\begin{document}


\begin{abstract}
Let $X$ be a finite set and let $G$ be a finite group acting on $X$. The group action splits $X$ into disjoint orbits. The Burnside process is a Markov chain on $X$ which has a uniform stationary distribution when the chain is lumped to orbits. We consider the case where $X = [k]^n$ with $k \geq n$ and $G = S_k$ is the symmetric group on $[k]$, such that $G$ acts on $X$ by permuting the value of each coordinate. The resulting Burnside process gives a novel algorithm for sampling a set partition of $[n]$ uniformly at random. We obtain bounds on the mixing time and show that the chain is rapidly mixing. For the case $k < n$, the algorithm corresponds to sampling a set partition of $[n]$ with at most $k$ blocks, and we obtain a mixing time bound which is independent of $n$. Along the way, we obtain explicit formulas for the transition probabilities and bounds on the second largest eigenvalue for both the original process and the lumped chain. 
\end{abstract}

\maketitle


\section{Introduction} 

Markov chain Monte Carlo algorithms are used as computational methods for sampling from complicated probability distributions and are a mainstay in a wide range of scientific fields. 
Let $X$ be a finite set and let $\pi$ be a probability distribution on $X$. The computational problem is to sample from $X$ according to the distribution $\pi$. Typically $X$ is large so that it is difficult to sample from $\pi$ directly. The {\em Markov chain Monte Carlo} method provides an algorithmic solution to this problem by constructing an irreducible and aperiodic Markov chain, $K$, on $X$ whose stationary distribution is $\pi$. To sample from $X$ according to $\pi$, we start at an arbitrary state $x_0 \in X$ and run the Markov chain for $T$ many steps, then output the final state. Since $K$ converges to $\pi$ by construction, taking $T$ to be sufficiently large ensures that the distribution of the final state will be close to $\pi$. 

The efficiency of these algorithms rely on the {\em mixing time} of the Markov chain, which is the number of steps needed to get sufficiently close to the stationary distribution. 
Since $X$ is typically large, an efficient algorithm should have a mixing time that is much smaller than $|X|$. 
Obtaining mixing time bounds is an active research area and the mixing time analyses for many Monte Carlo algorithms remain open problems. 
We refer the reader to Diaconis \cite{Dia08, Dia13} for wonderful and accessible surveys on Markov chain Monte Carlo. 

\subsection{The Burnside Process}

Let $X$ be a finite set and let $G$ be a finite group acting on $X$. The group action splits $X$ into disjoint {\em orbits}. Let $O_x = \{gx : g \in G\}$ be the orbit containing $x \in X$. 
The {\em Burnside process} is a Markov chain developed by Jerrum \cite{Jer93} as a practical and novel way of choosing an orbit uniformly at random. 
This provides an algorithm for the computational problem of sampling from {\em unlabeled structures}, which are combinatorial structures modulo a group of symmetries. 

Let $X_g = \{ y \in X : gy = y\}$ be the {\em fixed set} of $g \in G$ and let $G_x = \{g \in G : gx = x\}$ be the {\em stabilizer} of $x \in X$. 
The Burnside process is a Markov chain on $X$ whose transition between states $x$ and $y$ is: 
\begin{itemize}
\item From $x \in X$, choose $g \in G_x$ uniformly at random. 
\item From $g$, choose $y \in X_g$ uniformly at random. 
\end{itemize}
The transition matrix is given by
\bal
K(x,y) = \sum_{g \in G_x \cap G_y} \frac{1}{|G_x|} \frac{1}{|X_g|}
\eal
and the stationary distribution is $\pi(x) = \frac{1}{z |O_x|}$ where $z$ is the number of orbits. For $X = \Oc_1 \cup \Oc_2 \cup \dotsb \cup \Oc_z$, with $\{\Oc_k\}_{k=1}^z$ the disjoint orbits of $X$, and $\{X_t\}_{t=0}^\infty$ a Markov chain on $X$, the {\em lumped chain} or {\em projection} is the chain $\{Y_t\}_{t=0}^\infty$, given by $Y_t = a$ if $X_t \in O_a$. 
It follows that the Burnside process has a uniform stationary distribution when lumped into orbits. 

We conclude this section with a survey of previous results. The Burnside process was introduced by Jerrum in \cite{Jer93} as an algorithm for sampling unlabeled structures. He showed that for many groups, the Burnside process is {\em rapidly mixing}, which means the mixing time is upper bounded by a polynomial in the input size. 
However, Goldberg and Jerrum \cite{GJ02} showed that the Burnside process is not rapidly mixing in general. 

Aldous and Fill \cite{AF02} considered the case where $X = [k]^n$, the set of $n$-tuples with entries from $[k]$, and $G = S_n$ acts on $X$ by permuting coordinates, 
and obtained a mixing time upper bound of order $k \log n$. 
Subsequently, Diaconis \cite{Dia05} found connections to Bose-Einstein configurations and obtained an upper bound which is independent of $n$, for $k$ fixed or growing slowly with $n$. 
More recently, Diaconis and Zhong \cite{DZ21} obtained sharp mixing time bounds for the special case $k=2$. In another direction, Rahmani \cite{Rah20} studied the {\em commuting chain}, 
where $X = G$ is a finite group acting on itself by conjugation, and showed that the chain is rapidly mixing for various groups. 
In computer science, Holtzen, Millstein, and Van den Broeck \cite{HMV20} found an application of the Burnside process to an approximate inference algorithm called the {\em orbit-jump Markov chain Monte Carlo}. 

The most general result is due to Chen \cite{Che06}, who showed that for any Burnside process, the mixing time is upper bounded by $|G|$ for the original chain, and is upper bounded by $|X|$ for the lumped chain. Although these bounds work in general, they are often not strong enough to prove rapid mixing for many special cases. 
Thus the quantitative analysis of the mixing time for the Burnside process, in full generality, remains an open problem. 


\subsection{Main Results} \label{MainResults}

Let $[n] := \{1, 2, \ldots, n\}$. A {\em set partition} of $[n]$ is a set of nonempty subsets of $[n]$ such that $[n]$ is a disjoint union of the subsets. We refer to these subsets as {\em blocks}. 
Let $\Pi_n$ denote the collection of all set partitions of $[n]$. The number of set partitions of $[n]$ is given by $B_n$, the $n$th {\em Bell number}. 
The {\em Stirling numbers of the second kind} $\stirling{n}{k}$ count the number of set partitions of $[n]$ into $k$ nonempty subsets. These numbers are connected through the sum $B_n = \sum_{k=0}^n \stirling{n}{k}$. 

We consider a variant of the Burnside process studied by Aldous and Fill \cite{AF02}, Diaconis \cite{Dia05}, and Diaconis and Zhong \cite{DZ21}, first posed by Rahmani in his thesis \cite{Rah21}. Let $X = [k]^n$, where $k \geq n \geq 1$, and let $G = S_k$. 
Consider the group action of $S_k$ on $[k]^n$ which permutes coordinates by value. That is, if $u = (u_1, \ldots, u_n) \in [k]^n$ and $\sigma \in S_k$, then $\sigma u = (\sigma(u_1), \ldots, \sigma(u_n))$. 
Let $X_\sigma = \{v \in X : \sigma v = v\}$ and $G_u = \{\sigma \in S_k : \sigma u = u\}$. The {\em Burnside process on $[k]^n$} is a Markov chain whose transition between states $u,v \in [k]^n$ can be described explicitly as follows. 
\begin{itemize}
\item From $u \in [k]^n$, identify the set of distinct values $J \subseteq [k]$ which appear in $u$. Choose a uniformly random permutation $\sigma \in S_k$ conditioned to have fixed points at positions given by $J$. 
Let $\FP(\sigma)$ be the set of fixed points of $\sigma$. 
\item From $\sigma$, let $v \in [k]^n$ be obtained by setting each coordinate independently with a uniform choice in $\FP(\sigma)$.
\end{itemize}

Observe that the orbit of $u \in [k]^n$ is determined by the sets of coordinates whose values are equal. Indeed, if $u_i = u_j$, then $\sigma(u_i) = \sigma(u_j)$ for all $\sigma \in S_k$. 
Given $u$, we can define a corresponding set partition $x \in \Pi_n$ such that $i,j \in [n]$ are in the same block of $x$ if and only if $u_i = u_j$. Thus if $x$ is the set partition corresponding to $u$, then $x$ also corresponds to $\sigma u$. 
It follows that the orbits are indexed by the set partitions of $[n]$. When lumped into orbits, the Burnside process on $[k]^n$ defines a Markov chain on set partitions of $[n]$ with uniform stationary distribution. 
We refer to this lumped chain as the {\em Burnside process on set partitions} or the {\em Burnside process on $\Pi_n$}. 

The Burnside process on $[k]^n$ provides a novel Markov chain Monte Carlo algorithm for sampling a set partition of $[n]$ uniformly at random. 
Starting from an arbitrary state $u \in [k]^n$, we simulate the Burnside process for a sufficiently large number of steps and then output the final state. 
Then the set partition corresponding to the final state will be approximately uniformly distributed. Our main result is an upper bound on the mixing time of the Burnside process on $[k]^n$. 

\begin{theorem} \label{maintheorem}
Let $K(u,v)$ be the transition matrix of the Burnside process on $[k]^n$, where $k \geq n \geq 1$. Let $\pi(u) = \frac{1}{B_n |O_u|}$ be its stationary distribution. Then
\bal
\|K_u^t - \pi\|_{TV} \leq n\left(1 - \frac{1}{2k}\right)^t.
\eal
The upper bound is uniform in the starting state $u \in [k]^n$. Therefore the mixing time of $K$ satisfies
\bal
t_\mix(\epsilon) \leq \left\lceil 2k\log\left( \frac{n}{\epsilon} \right) \right\rceil.
\eal
\end{theorem}

This upper bound shows that the Burnside process on $[k]^n$ is rapidly mixing. In the case where $1 \leq k < n$, the Burnside process on $[k]^n$ gives a Markov chain Monte Carlo algorithm for sampling a set partition of $[n]$ of at most $k$ blocks uniformly at random. 
In this regime, we obtain a stronger mixing time upper bound of $t_\mix(\epsilon) \leq \left\lceil (k-1)! \log\left( \epsilon^{-1} \right) \right\rceil$
in Theorem \ref{minorizationbound}, which shows that for fixed $k$, the mixing time is independent of the number of coordinates, $n$. 
Along the way, we also obtain:
\begin{itemize}[leftmargin=0.5cm]
\item Theorem \ref{setpartitiontransition}: explicit formulas for the transition probabilities of the Burnside process on $\Pi_n$,
\item Proposition \ref{eigenvalueUBoriginal}: an upper bound on the second largest eigenvalue of the Burnside process on $[k]^n$, 
\item Theorem \ref{eigenvaluebounds}: upper and lower bounds on the second largest eigenvalue of the Burnside process on $\Pi_n$.
\item Proposition \ref{relaxationtimeUBLB}: upper and lower bounds on the mixing time of the Burnside process on $\Pi_n$. 
\end{itemize}

The main difficulty in proving our results stems from the complicated formulas for the transition probabilities of both the original process and the lumped chain. 
These make spectral methods intractable. To overcome these difficulties, we instead use probabilistic and geometric techniques. Our bounds are proved using coupling and minorization, 
which turn out to be useful techniques for chains which make large jumps in a single step. 


\subsection{Outline}

The paper is organized as follows. In Section \ref{SectionPreliminaries} we give an overview of the techniques that we use for our mixing time bounds. 
Then in Section \ref{SectionTransitionProbabilities}, we derive the transition matrix for the Burnside process on $\Pi_n$. In Section \ref{SectionMixingTimeBounds} we use coupling and minorization to obtain upper bounds on the mixing time of the Burnside process on $[k]^n$.
We then use geometric methods in Section \ref{GeometricBounds} to obtain bounds on the second largest eigenvalue of the lumped process, which lead to mixing time bounds for the lumped chain via the relaxation time. We conclude with some final remarks in Section \ref{SectionFinalRemarks}.


\subsection*{Acknowledgments} 
We thank Jason Fulman for suggesting the problem and providing useful references, Persi Diaconis for encouraging us to strengthen our mixing time upper bound, and Richard Arratia for a wonderful conversation on probabilistic divide-and-conquer. 
We also thank John Rahmani, Sanat Mulay, Clemens Oszkinat, Apoorva Shah, Peter Kagey, and Ivan Feng for helpful discussions. 
Finally we thank an anonymous referee for many helpful comments and suggestions.


\section{Preliminaries} \label{SectionPreliminaries}

This section gives necessary background on Markov chain convergence, the Burnside process, and auxiliary variables algorithms, and presents the various techniques which we use to obtain bounds on mixing times. 
We aim to give an accessible introduction, along with several references to the literature.

\subsection{Markov Chains and Mixing Times}

Let $X$ be a finite set. A Markov chain is specified by a {\em transition matrix} $K(x,y) \geq 0$ with $\sum_y K(x,y) = 1$, so that $K(x,y)$ is the probability of moving from $x$ to $y$ in one step. 
Under mild conditions, there exists a unique {\em stationary distribution} $\pi(x) \geq 0$ with $\sum_x \pi(x) = 1$, such that $\sum_x \pi(x) K(x,y) = \pi(y)$. 

Let $K^t(x,y)$ be the probability of moving from $x$ to $y$ in $t$ steps. 
The {\em fundamental theorem of Markov chains} states that if $K$ is irreducible and aperiodic, then $K_x^t(y) := K^t(x,y) \to \pi(y)$ as $t \to \infty$. 
The distance to stationarity can be measured by the {\em total variation distance}
\bal
\|K_x^t - \pi\|_{TV} = \max_{A \subseteq X} |K^t(x,A) - \pi(A)| = \frac{1}{2} \sum_y |K^t(x,y) - \pi(y)|. 
\eal

Let $K$ be an irreducible and aperiodic Markov chain on $X$ with stationary distribution $\pi$, so that $\|K_x^t - \pi\|_{TV} \to 0$ as $t \to \infty$ for all $x \in X$. Let $d(t) = \max_{x \in X} \|K_x^t - \pi\|_{TV}$
be the {\em distance function}. Fix $\epsilon > 0$. The {\em mixing time} of the Markov chain $K$ is defined to be
\bal
t_{\mix}(\epsilon) = \min\{ t : d(t) \leq \epsilon \}. 
\eal

\subsection{Auxiliary Variables}

The {\em auxiliary variables} Markov chain was introduced by Edwards and Sokal \cite{ES88} as an abstraction of the Swendsen-Wang algorithm \cite{SW87}. 
Auxiliary variables gives a method for constructing a reversible Markov chain with stationary distribution $\pi$, and is related to several classes of unifying algorithms, 
which include {\em data augmentation} \cite{TW87} and {\em hit-and-run} \cite{AD07}. These are non-local chains which make large moves in a single step. Thus they typically mix much faster than local or diffusive chains. 

Let $\pi(x)$ be a probability distribution on $X$. Let $I$ be a set of auxiliary variables. For all $x \in X$, let $w_x(i)$ be a probability distribution on $I$, which gives the probability of moving from $x$ to $i$. 
These define a joint distribution $f(x,i) = \pi(x) w_x(i)$ and a marginal distribution $m(i) = \sum_x f(x,i)$. For all $i \in I$, we specify a Markov chain $K_i(x,y)$ with reversing measure $f(x \mid i) = f(x,i)/m(i)$. 
The auxiliary variables Markov chain transitions from states $x$ to $y$ as follows: from $x \in X$, choose $i \in I$ with probability $w_x(i)$, then from $i$, choose $y \in X$ with probability $K_i(x,y)$. 
The transition matrix is $K(x,y) = \sum_i w_x(i) K_i(x,y)$, and it is straightforward to check that $K$ is reversible with respect to $\pi$. 

The Burnside process is the special case obtained by setting $I = G$, $w_x(i)$ the uniform distribution on $G_x$, and $K_g(x,y)$ the uniform distribution on $X_g$. 
The motivation for its study is that the group structure should make mixing time analysis more tractable.

\subsection{Coupling}

The method of coupling is one of the most powerful probabilistic techniques for bounding mixing times. 
A {\em coupling} of two probability distributions $\mu, \nu$ is a Markov chain $(X_t,Y_t)$ defined on a single probability space such that the marginal distribution of $X_t$ is $\mu$ and the marginal distribution of $Y_t$ is $\nu$. The {\em coupling lemma} states that for all couplings $(X,Y)$ of $\mu$ and $\nu$, the total variation distance can be upper bounded by the probability that the two couplings are not equal, $\|\mu - \nu\|_{TV} \leq P(X \neq Y)$. Moreover, there always exists an optimal coupling which achieves equality. 

\subsection{Minorization}

A Markov chain $K$ on a state space $X$ satisfies a {\em minorization} condition if there exists a probability measure $\nu$ on $X$, a positive integer $t_0$, and $\delta > 0$ such that 
\bal
P^{t_0}(x,A) \geq \delta \nu(A), 
\eal
for all $x \in X$ and for all measurable subsets $A \subseteq X$. Minorization is closely related to Harris recurrence \cite{AN78} and we refer the reader to \cite{RR04} for an accessible survey on minorization in Markov chains. 
The following proposition gives a bound on the distance to stationarity for chains which satisfy a minorization condition. 

\begin{proposition}[\cite{Ros95a}, Proposition 2] \label{minorizationboundlemma}
Let $K$ be the transition matrix of a Markov chain on a finite state space $X$ with stationary distribution $\pi$. Suppose there exists a probability distribution $\nu$ on $X$, 
a positive integer $t_0$, and some $\delta > 0$ such that $K^{t_0}(x,A) \geq \delta \nu(A)$ for all $x \in X$ and for all measurable subsets $A \subseteq X$. Then for all $t$,
\bal
\|K_x^t - \pi\|_{TV} \leq (1 - \delta)^{\lfloor t/t_0 \rfloor}
\eal
\end{proposition}

Minorization has been used to obtain convergence rates on Markov chain Monte Carlo algorithms, 
for example by Rosenthal for the data augmentation algorithm \cite{Ros93} and for the Gibbs sampler \cite{Ros95b}. 
More recently, it was used by Diaconis \cite{Dia05} and Rahmani \cite{Rah20} to obtain mixing time bounds on Burnside processes. 

\subsection{Spectral Gap} 

Let $K$ be a reversible Markov chain. By the spectral theorem, $K$ has real eigenvalues, with $1 = \lambda_0 \geq \lambda_1 \geq \dotsb \geq \lambda_{|X|-1} \geq -1$. 
The {\em absolute spectral gap} is defined to be $1 - \lambda_*$, where $\lambda_* = \max\{\lambda_1, |\lambda_{|X| - 1}|\}$. If $K$ is also irreducible and aperiodic, then $1 - \lambda_* > 0$, and the {\em spectral gap} is defined as $1 - \lambda_1$. In practice, it suffices to consider $\lambda_1$ since the chain can be modified so that $\lambda_1 \geq |\lambda_{|X|-1}|$. One way is to consider the {\em continuous-time chain} with transition matrix $K$, where the waiting times between transitions are iid exponential rate $1$ random variables. The transition kernel is
\bal
H_t(x,y) = e^{-t} \sum_{n = 0}^\infty \frac{t^n K^n(x,y)}{n!},
\eal
for all $x,y \in X$. The following proposition shows the relationship between $H_t$ and the spectral gap of $K$. 

\begin{proposition}[\cite{Sal97}, Theorem 2.1.7] \label{spectralgapheatkernelrelation}
Let $K$ be an irreducible and reversible Markov chain, and let $H_t$ be transition kernel for the continuous time chain with transition matrix $K$. Then the spectral gap of $K$ satisfies
\bal
1 - \lambda_1 = \lim_{t \to \infty} \frac{-1}{t} \log\left( 2 \max_{x \in X} \| H_t(x,\cdot) - \pi(\cdot) \|_{TV} \right).
\eal
\end{proposition}

\subsection{Geometric Bounds}

Let $K$ be the transition matrix of an irreducible and reversible Markov chain on a finite state space $X$ with stationary distribution $\pi$. Let $Q(x,y) = \pi(x)K(x,y)$ for all $x,y \in X$ be the reversing measure. Consider the graph with vertex set $X$ and $(x,y)$ an edge if and only if $Q(x,y) > 0$. This is the {\em underlying graph} of the Markov chain $K$. 
Note that this graph may contain self-loops, is always connected, and uniquely specifies the chain. 
For all pairs of distinct vertices $x,y \in X$, construct a path $\gamma_{xy}$ from $x$ to $y$, called the {\em canonical path}. Let $\Gamma$ denote the collection of all canonical paths.

Diaconis and Stroock \cite{DS91} and Sinclair \cite{Sin92} obtained bounds on $\lambda_1$ in terms of geometric quantities arising in the underlying graph. 
The following {\em Poincar\'{e} bound} is based on the Poincar\'{e} inequality and uses the method of canonical paths. 

\begin{proposition}[\cite{DS91}, Proposition 1'] \label{PoincareBound}
For an irreducible Markov chain, the second largest eigenvalue satisfies $\lambda_1 \leq 1 - (1/\bar{\kappa})$ where 
\bal
\bar{\kappa} = \max_{e} \frac{1}{Q(e)} \sum_{\gamma_{xy} \ni e} |\gamma_{xy}| \pi(x) \pi(y)
\eal
and $|\gamma_{xy}|$ denotes the number of edges in the path $\gamma_{xy}$. 
\end{proposition}

The geometric quantity $\bar{\kappa}$ is a measure of {\em bottlenecks} in the graph, and will be small if it is possible to pick paths that do not pass through any one edge too often. 
Let $S \subseteq X$ be a proper subset and define
\bal
Q(S, S^c) = \sum_{x \in S} \sum_{y \in S^c} Q(x,y) = \sum_{x \in S} \sum_{y \in S^c} \pi(x) K(x,y).
\eal
Bounds on $\lambda_1$ can also been derived in terms of the {\em conductance}
\bal
\Phi = \min_{\pi(S) \leq 1/2} \frac{Q(S, S^c)}{\pi(S)}.
\eal
This quantity measures the flow out of the set $S$ when the chain is at stationarity. If $\Phi$ is large for all $S$, then the chain should converge to $\pi$ quickly. 
Sinclair and Jerrum \cite{SJ89} obtained the {\em Cheeger bound} $\lambda_1 \leq 1 - \frac{\Phi^2}{2}$ based on the Cheeger inequality. Alternatively \cite{JS89, SJ89}, they obtained the bound $\lambda_1 \leq 1 - \frac{1}{8\eta^2}$, where $\eta = \max_{e} Q(e)^{-1} \sum_{\gamma_{xy} \ni e} \pi(x) \pi(y)$. 
Diaconis and Stroock \cite{DS91} observed that for many Markov chains, the Poincar\'{e} bound beats the Cheeger bound. 
Fulman and Wilmer \cite{FW99} showed that this is true for simple random walk on trees and random walks on finite groups with a symmetric generating set. 

Next we turn to a lower bound on $\lambda_1$. The first inequality is written in a form given by Ingrassia \cite{Ing94}. 

\begin{proposition}[\cite{DS91}, Proposition 6] \label{eigenvalueLB}
Let $S$ be a proper subset of $X$. Then
\bal
\lambda_1 \geq 1 - \frac{Q(S, S^c)}{\pi(S) \pi(S^c)} \geq 1 - 2\Phi. 
\eal
\end{proposition}

Finally recall that the relaxation time of a reversible Markov chain is related to the spectral gap by $t_{\rel} = (1 - \lambda_1)^{-1}$. 
The following proposition gives bounds on the mixing time in terms of the relaxation time. 

\begin{proposition}[\cite{LPW09}, Theorems 12.4 and 12.5] \label{relaxationtimebounds}
Let $K$ be the transition matrix of a reversible, irreducible, and aperiodic Markov chain with state space $X$, and let $\pi_{\min} := \min_{x \in X} \pi(x)$. Then
\bal
(t_\rel - 1)\log\left( \frac{1}{2\epsilon} \right) \leq t_{\mix}(\epsilon) \leq \left\lceil t_\rel \log \left( \frac{1}{\epsilon \, \pi_{\min}} \right) \right\rceil.
\eal
\end{proposition}

The upper bound follows from Proposition 3 in \cite{DS91} while the lower bound can be derived from a discrete time version of Proposition 8 in \cite{Ald82}.

%
%


\section{Transition Probabilities} \label{SectionTransitionProbabilities}

In this section, we compute the transition probabilities for the Burnside process on $[k]^n$ and use this to derive the transition probabilities for the Burnside process on $\Pi_n$. 

Let $K$ be the transition matrix for the Burnside process on $[k]^n$ and let $u,v \in [k]^n$. Let $j_u$ be the number of distinct entries of $u$, $j_v$ the number of distinct entries of $v$, and $j$ the number of distinct entries when the entries of $u$ and $v$ are combined. In Theorem 5.2.1 of their thesis \cite{Rah21}, Rahmani showed that 
\bal
K(u,v) = \frac{(k-j)!}{(k-j_u)!} E\left[ \frac{1}{(Y+j)^n} \right]
\eal
where $Y$ is the number of fixed points of a uniformly random permutation $\sigma \in S_{k-j}$. Observe that $K(u,v)$ is completely determined by the values of $j_u$ and $j$. Moreover, it is straightforward to check that the $K$ is irreducible, aperiodic, and reversible with respect to $\pi(u) = \frac{1}{B_n |O_u|} = \frac{(k-j_u)!}{B_n k!}$, where $|O_u| = \frac{k!}{(k-j_u)!}$ by the Orbit-Stabilizer Theorem. 

Our new contribution is an explicit formula for the transition matrix, $\bar{K}$, of the Burnside process on $\Pi_n$. 
Let $u,v \in [k]^n$ and let $j_u, j_v, j$ be defined as above. Suppose $j_u \leq j_v$. 
Then the probability of transitioning from the orbit $O_u$ to $O_v$ is given by $K(O_u, O_v) = K(u, O_v)$. Since orbits are determined by entries with equal values, 
it suffices to add $K(u,v)$ as $v$ ranges over all states in the orbit $O_v$, so that $K(O_u, O_v) = \sum_{y \in O_v} K(u,y)$. 

Next note that $j_u \leq j_v \leq j \leq j_u + j_v$. We count the possible values that can occur in the distinct entries of $v$. 
There are $\binom{j_u}{j_u + j_v - j}$ ways to pick the values in $u$ which also appear in $v$. Then there are $\binom{k-j_u}{j-j_u}$ to pick the remaining values from $[k]$ which appear in $v$. 
This gives $j_v$ distinct values from $[k]$ which appears in $v$. Finally there are $j_v!$ ways to distribute these values among the $j_v$ entries where the distinct values occur. Combining the above yields the transition probabilities for the lumped chain,
\bal
K(O_u, O_v) &= K(u, O_v) = \sum_{j=j_v}^{j_u + j_v} j_v! \binom{j_u}{j_u + j_v - j} \binom{k-j_u}{j - j_u} K(u,v) \\
&= \sum_{j=j_v}^{j_u + j_v} \frac{j_u! j_v!}{(j-j_u)! (j-j_v)! (j_u + j_v - j)!} E\left[ \frac{1}{(Y+j)^n} \right]. 
\eal

We remark that $K(O_u, O_v)$ is completely determined by the values of $j_u$ and $j_v$. Moreover note that if $j_u \geq j_v$, then $K(O_u, O_v) = K(O_v, O_u)$, since the transition probability is symmetric in $j_u$ and $j_v$.
Our derivation establishes the following.

\begin{theorem} \label{setpartitiontransition} 
Let $x,y \in \Pi_n$. Let $j_x$ and $j_y$ be the number of blocks of $x$ and $y$, respectively. Suppose $j_x \leq j_y$. Then the Burnside process on $\Pi_n$ has transition matrix
\bal
\bar{K}(x,y) = \sum_{j=j_y}^{j_x + j_y} \frac{j_x! j_y!}{(j-j_x)! (j-j_y)! (j_x + j_y - j)!} E\left[ \frac{1}{(Y+j)^n} \right]
\eal
where $Y$ is the number of fixed points of a uniformly random permutation $\sigma \in S_{k-j}$. If $j_x > j_y$, then $\bar{K}(x,y) = \bar{K}(y, x)$. 
\end{theorem}

It is clear that $\bar{K}$ is irreducible, aperiodic, and reversible with respect to the uniform distribution on set partitions $\bar{\pi}(x) = 1/B_n$. It will be useful to represent $\bar{K}$ in a canonical form as follows. 
Let $x_1, x_2, \ldots, x_{B_n}$ denote the $B_n$ set partitions of $[n]$ {\em ordered by the number of blocks}, so that $x_1$ is the unique set partition consisting of a single block, $x_2, \ldots, x_{2 + \stirling{n}{2}}$ are the $\stirling{n}{2}$ many set partitions consisting of two blocks 
(listed in a fixed but arbitrary order), and so on, until $x_{B_n}$, the unique set partition consisting of $B_n$ singleton blocks. 
Finally, let $(\bar{K}_{ij})_{1 \leq i,j \leq B_n}$ be the $B_n \times B_n$ transition matrix such that $\bar{K}_{ij} := \bar{K}(x_i, x_j)$. We refer to $(\bar{K}_{ij})$ as the transition matrix of $\bar{K}$ written in {\em canonical form}. \\

\noindent {\bf Remark.} Observe that $\bar{K}(x,y)$ is completely determined by the block sizes $j_x, j_y$. 
Thus we may write $\bar{K}(x,y) := \bar{K}(j_x, j_y)$, and we use $\bar{K}(j_x, j_y)$ whenever it is more convenient to describe the transition probability in terms of the block sizes. 
Moreover observe that $(\bar{K}_{ij})$ is a symmetric matrix. In fact, it is a symmetric block matrix of the form
\bal
(\bar{K}_{ij}) = \begin{pmatrix}
A_{11} & A_{12} & \cdots & A_{1n} \\
A_{21} & A_{22} & \cdots & A_{2n} \\
\vdots & \vdots & \ddots & \vdots \\
A_{n1} & A_{n2} & \cdots & A_{nn}
\end{pmatrix}
\eal
where $A_{ij}$ is an $\stirling{n}{i} \times \stirling{n}{j}$ submatrix whose entries are all equal to $\bar{K}_{ij}$, with $A_{ji} = A_{ij}^T$ for all $1 \leq i,j \leq n$. 
It follows that we can further project the lumped chain $\bar{K}$ to a Markov chain $L$, given by the $n \times n$ matrix with $L_{ij} = \stirling{n}{j}\bar{K}_{ij}$, 
where $L_{ij}$ is the transition probability of going from a set partition with $i$ blocks to a set partition with $j$ blocks. 
However note that $(L_{ij})$ is no longer a symmetric matrix.

For example, let $k \geq n = 2$. Let $!k$ denote the number of derangements of $[k]$. The transition matrix for the Burnside process on $\Pi_2$ is
\bal
\begin{pmatrix} 1 - \frac{!k}{k!} & \frac{!k}{k!} \\ \frac{!k}{k!} & 1 - \frac{!k}{k!} \end{pmatrix}
\eal
with stationary distribution $\pi = (1/2, 1/2)$. The eigenvalues are $\lambda = 1$ and $\lambda = 1 - 2\left(\frac{!k}{k!}\right)$. 
It follows that $c_\epsilon \leq t_{\mix}(\epsilon) \leq C_\epsilon$, where $c_\epsilon, C_\epsilon$ are constants independent of $k$. 

For $k \geq n = 3$, transition matrix for the Burnside process on $\Pi_3$ can be written in block matrix form as
\bal
\left(
\begin{array}{c; {2pt/2pt} c; {2pt/2pt} c}
S_{3,k,1} + S_{3,k,2} & 2S_{3,k,2} + S_{3,k,3} & 3S_{3,k,3} + S_{3,k,4} \\ \hdashline[2pt/2pt]
2S_{3,k,2} + S_{3,k,3} & 2S_{3,k,2} + 4S_{3,k,3} + S_{3,k,4} & 6S_{3,k,3} + 6S_{3,k,4} + S_{3,k,5} \\ \hdashline[2pt/2pt]
3S_{3,k,3} + S_{3,k,4} & 6S_{3,k,3} + 6S_{3,k,4} + S_{3,k,5} & 6S_{3,k,3} + 18S_{3,k,4} + 9S_{3,k,5} + S_{3,k,6}
\end{array}
\right)
\eal
where the dashed lines denote the submatrices $A_{ij}$ defined above and $S_{n,k,j} := E\left[ \frac{1}{(Y + j)^n} \right]$
where $Y$ is the number of fixed points of a uniformly random permutation $\sigma \in S_{k-j}$. Even for this small case the characteristic polynomial does not admit a nice formula, and we are unable to diagonalize $\bar{K}$ in general. On the other hand, setting $k = 3$ gives
\bal
\begin{pmatrix} 
5/9 & 1/9 & 1/9 & 1/9 & 1/9 \\
1/9 & 2/9 & 2/9 & 2/9 & 2/9 \\
1/9 & 2/9 & 2/9 & 2/9 & 2/9 \\
1/9 & 2/9 & 2/9 & 2/9 & 2/9 \\
1/9 & 2/9 & 2/9 & 2/9 & 2/9 
\end{pmatrix}
\eal
so that the eigenvalues are $\lambda = 1$, $\lambda = 4/9$, and $\lambda = 0$ (with multiplicity $3$). One can then upper bound the mixing time in terms of the second largest eigenvalue. \\

We conclude this section with a lemma which shows that the total variation distance between $K^t$ and $\pi$ is lower bounded by the total variation distance between $\bar{K}^t$ and $\bar{\pi}$. 

\begin{lemma} \label{totalvariationdistancerelation}
Let $K$ and $\pi$ be the transition matrix and stationary distribution of the Burnside process on $[k]^n$. Let $\bar{K}$ and $\bar{\pi}$ be the transition matrix and stationary distribution of the Burnside process on $\Pi_n$. 
Let $x \in \Pi_n$ be the set partition corresponding to $u \in [k]^n$. Then
\bal
\|K_u^t - \pi\|_{TV} \geq \|\bar{K}_x^t - \bar{\pi}\|_{TV}.
\eal
\end{lemma}

\begin{proof}
Recall that we can write $[k]^n = \bigcup_{i=1}^{B_n} \Oc_i$, where $\{\Oc_i\}_{i=1}^{B_n}$ are the orbits. Then 
\bal
\|K_u^t - \pi\|_{TV}  &= \frac{1}{2} \sum_{v \in [k]^n} |K_u^t(v) - \pi(v)| = \frac{1}{2} \sum_{i=1}^{B_n} \sum_{v \in \Oc_i} |K_u^t(v) - \pi(v)| \\
&\geq \frac{1}{2} \sum_{i=1}^{B_n} \left| \sum_{v \in \Oc_i} \left( K_u^t(v) - \frac{1}{B_n |\Oc_i|} \right) \right| = \frac{1}{2} \sum_{i=1}^{B_n} \left| \bar{K}_x^t(\Oc_i) - \frac{1}{B_n} \right| =  \|\bar{K}_x^t - \bar{\pi}\|_{TV},
\eal 
where in the second line above we used the triangle inequality.
\end{proof}


\section{Mixing Time Upper Bounds} \label{SectionMixingTimeBounds}

In this section we prove our main result, Theorem \ref{maintheorem}, the upper bound on the mixing time of the Burnside process on $[k]^n$, for $k \geq n \geq 1$, using the method of coupling. 
We then prove a stronger upper bound in Theorem \ref{minorizationbound} for the $1 \leq k < n$ regime using minorization. 

\subsection{Coupling Bound} 

Our upper bound on the mixing time is obtained through the coupling method. We start with two useful lemmas which we use for our coupling constructions. 

\begin{lemma} [\cite{AF02}, Lemma 12.4] \label{uniformperminduced}
Given finite sets $F^1, F^2$ we can construct (for $u = 1,2$) a uniformly random permutation $\sigma^u$ of $F^u$ with cycles $(C_j^u; j \geq 1)$, where the cycles are labeled such that 
\bal
C_j^1 \cap F^1 \cap F^2 = C_j^2 \cap F^1 \cap F^2
\eal
for all $j$, where in the equality the $C_j^u$ are interpreted as sets. 
\end{lemma}

%

\begin{lemma} [\cite{Jer03}, Lemma 4.10] \label{couplingjointdist}
Let $U$ be a finite set and let $A,B \subseteq U$. Let $Z_a, Z_b$ be random variables taking values in $U$. Then there exists a joint distribution on $(Z_a, Z_b)$ such that $Z_a$ is uniform on $A$, $Z_b$ is uniform on $B$, and 
\bal
P(Z_a = Z_b) = \frac{|A \cap B|}{\max\{|A|, |B|\}}. 
\eal
\end{lemma}

%

We are now in a position to prove our main theorem. 

\begin{proof}[Proof of Theorem \ref{maintheorem}]
Let $u,v \in [k]^n$. 
Let $J_u$ and $J_v$ be the set of distinct entries in $u$ and $v$, respectively. We construct a coupling $(U,V)$ of a step of the chain, started at $(u,v)$. We consider two cases. 

{\bf Case 1:} Suppose $|J_u \cap J_v| \neq 0$. Define the sets $A := [k] \setminus J_u$ and $B := [k] \setminus J_v$. Let $\sigma$ be a uniformly random permutation of the set $A \cup B$. 
Let $\sigma^1 := \left. \sigma \right|_A$ and $\sigma^2 := \left. \sigma \right|_B$ be the induced permutation of $\sigma$ on $A$ and $B$, respectively. 
By Lemma \ref{uniformperminduced}, $\sigma^1$ is a uniformly random permutation of $A$ and $\sigma^2$ is a uniformly random permutation of $B$. 
Finally define $\sigma^u \in S_k$ to be the permutation with fixed points at indices given by $J_u$ and such that the remaining indices $A$ are permuted according to $\sigma^1$. 
Similarly define $\sigma^v \in S_k$ to be the permutation with fixed points at indices given by $J_v$ and such that the remaining indices $B$ are permuted according to $\sigma^2$. 
By construction, $\sigma^u$ (respectively, $\sigma^v$) is a uniformly random permutation of $S_k$ conditioned to have fixed points at indices given by $J_u$ (respectively, $J_v$). This completes the first substep of the chain. 

Let $\FP(\sigma^u)$ (respectively, $\FP(\sigma^v)$) be the set of fixed points of $\sigma^u$ (respectively, $\sigma^v$). Let $\fp(\sigma^u) = |\FP(\sigma^u)|$ and $\fp(\sigma^v) = |\FP(\sigma^v)|$. 
Without loss of generality, suppose $\fp(\sigma^u) \geq \fp(\sigma^v)$. Since $|J_u \cap J_v| \neq 0$, we have that $|\FP(\sigma^u) \cap \FP(\sigma^v)| \geq 1$ by construction. 
Moreover, $\fp(\sigma^u) \leq k$ and $\fp(\sigma^v) \leq k$. 
By Lemma \ref{couplingjointdist} there exists a coupling $(U_i, V_i)$ such that $U_i$ is uniform on $\FP(\sigma^u)$, $V_i$ is uniform on $\FP(\sigma^v)$, 
and $P(U_i = V_i) = \frac{|\FP(\sigma^u) \cap \FP(\sigma^v)|}{|\FP(\sigma^u)|} \geq \frac{1}{k}$, for all $1 \leq i \leq n$. 
Let $U \in [k]^n$ (respectively, $V \in [k]^n$) be constructed by setting the $i$th coordinate equal to $U_i$ (respectively, $V_i$) independently for all $1 \leq i \leq n$. This completes the second substep of the chain. 

Therefore the coupled process $(U, V)$ satisfies $P(U_i \neq V_i) \leq \left(1 - \frac{1}{k} \right)$. 
See Example \ref{excouplingconstruction} Case 1 for an example of this coupling construction. 

{\bf Case 2:} Suppose $|J_u \cap J_v| = 0$. We construct a coupling $(U,V)$ of a step of the chain in exactly the same way as Case 1. However, since $|J_u \cap J_v| = 0$, 
we have by construction that $|\FP(\sigma^u) \cap \FP(\sigma^v)| \geq 1$ if and only if $\sigma \in S_k$ (as constructed in Case 1) has at least one fixed point. 

Let $A$ be the event that $\sigma$ has at least one fixed point, so that $A^c$ is the event that $\sigma$ has no fixed points. Using the well-known formula for derangements, we get 
\bal
P(A) = 1 - P(A^c) = 1 - \sum_{i=0}^{k-j_u-j_v} \frac{(-1)^i}{i!} \geq \frac{1}{2}.
\eal
Therefore by conditioning on the first substep of the chain, we get
\bal
P(U_i = V_i) = P(U_i = V_i, \text{$\sigma$ has at least one fixed point}) \geq \frac{1}{2k}, 
\eal
so that $P(U_i \neq V_i) \leq \left(1 - \frac{1}{2k} \right)$. See Example \ref{excouplingconstruction} Case 2 for an example of this coupling construction.

Taking the maximum over the two cases, we have constructed a coupling $(U,V)$ of a step of the chain for all pairs of initial states $(u,v)$ such that 
\bal
P(U_i \neq V_i) \leq \left(1 - \frac{1}{2k} \right). 
\eal

Let $U(t)$ and $V(t)$ be the states at time $t$ of the process started at $U(0)$ and $V(0)$, respectively. Let $U_i(t)$ and $V_i(t)$ be the values of the $i$th coordinate of $U(t)$ and $V(t)$, respectively. 
Then for all $1 \leq i \leq n$, 
\bal
P(U_i(t) \neq V_i(t)) \leq \left(1 - \frac{1}{2k} \right)^t
\eal
since the chain is time homogeneous, and
\bal
P(U(t) \neq V(t)) = P\left( \bigcup_{i=1}^n \{U_i(t) \neq V_i(t)\} \right) \leq \sum_{i=1}^n P(U_i(t) \neq V_i(t)) \leq n\left(1 - \frac{1}{2k}\right)^t,
\eal
by the union bound. Therefore by the coupling lemma,
\bal
\|K_u^t - \pi\|_{TV} \leq n\left(1 - \frac{1}{2k}\right)^t \leq n\exp\left(-\frac{t}{2k}\right)
\eal
uniformly over all starting states $u \in [k]^n$. Setting the right hand side equal to $\epsilon$ and solving for $t$ yields the desired mixing time upper bound $t_\mix(\epsilon) \leq \lceil 2k \log\left( \frac{n}{\epsilon} \right) \rceil$.
\end{proof}

Since the mixing time is upper bounded by a polynomial in $k$, Theorem \ref{maintheorem} shows that the Burnside process on $[k]^n$ is rapidly mixing. 

Before we proceed, we give an example of the coupling construction from the proof of Theorem \ref{maintheorem}.

\begin{example} \label{excouplingconstruction}
Let $n = k = 5$. 

{\bf Case 1:} Suppose that
\bal
u = (1,2,1,1,2), \quad v = (1, 1, 2, 3, 2). 
\eal
Then 
\bal
J_u = \{1, 2\}, \quad J_v = \{1, 2, 3\}, \quad |J_u \cap J_v| = |\{1, 2\}| = 2, 
\eal
so that we are in Case 1. Then
\bal
A = [5] \setminus J_u = \{3, 4, 5\}, \quad B = [5] \setminus J_v = \{4, 5\}, \quad A \cup B = \{3, 4, 5\}. 
\eal
We sample a uniformly random permutation of the set $\{3, 4, 5\}$. Suppose we sample $\sigma = (3 4)(5)$, written in cycle notation. Then the induced uniformly random permutations on $A$ and $B$, respectively, are
\bal
\sigma^1 = \left. \sigma \right|_A = (3 4)(5), \quad \sigma^2 = \left. \sigma \right|_B = (4)(5).
\eal
Define
\bal
\sigma^u = (1)(2)(3 4)(5), \quad \sigma^v = (1)(2)(3)(4)(5),
\eal
which by construction are uniformly random permutations of $S_5$ conditioned to have fixed points at indices given by $J_u$ and $J_v$, respectively. This completes the first substep. Next, the fixed point sets and their cardinalities are
\bal
\FP(\sigma^u) = \{1, 2, 5\}, \quad \fp(\sigma^u) = |\FP(\sigma^u)| = 3, \\
\FP(\sigma^v) = \{1, 2, 3, 4, 5\}, \quad \fp(\sigma^v) = |\FP(\sigma^v)| = 5. 
\eal
Note that $|\FP(\sigma^u) \cap \FP(\sigma^v)| = |\{1, 2, 5\}| = 3 \geq 1$. By Lemma \ref{couplingjointdist} there exists a coupling $(U_i, V_i)$ such that $U_i$ is uniform on $\FP(\sigma^u)$, $V_i$ is uniform on $\FP(\sigma^v)$, and 
\bal
P(U_i = V_i) = \frac{|\FP(\sigma^u) \cap \FP(\sigma^v)|}{|\FP(\sigma^v)|} = \frac{3}{5} \geq \frac{1}{5} = \frac{1}{k}.
\eal
Finally we construct $U$ (respectively, $V$) by setting the $i$th coordinate equal to $U_i$ (respectively, $V_i$) for all $1 \leq i \leq 5$. This completes the second substep. For example, suppose we obtain
\bal
U = (1, 2, 1, 2, 5), \quad V = (1, 2, 3, 1, 5).
\eal
Then by the coupling construction, the $i$th coordinates of $U$ and $V$ are equal with probability $3/5$.

{\bf Case 2:} Suppose
\bal
u = (1,2,1,1,2), \quad v = (3,3,4,5,4)
\eal
so that 
\bal
J_u = \{1,2\}, \quad J_v = \{3, 4, 5\}, \quad |J_u \cap J_v| = |\emptyset| = 0,
\eal
and we are in Case 2. Then
\bal
A = [5] \setminus J_u = \{3, 4, 5\}, \quad B = [5] \setminus J_v = \{1, 2\}, \quad A \cup B = \{1, 2, 3, 4, 5\},
\eal
and we sample a uniformly random permutation of the set $\{1, 2, 3, 4, 5\}$. Note that the probability this permutation has at least one fixed point is at least $1/2$. Suppose we sample $\sigma = (1 2)(3 4)(5)$, written in cycle notation. The induced uniformly random permutations on $A$ and $B$, respectively, are
\bal
\sigma^1 = \left. \sigma \right|_A = (34)(5), \quad \sigma^2 = \left. \sigma \right|_B = (12).
\eal
To complete the first substep, we define 
\bal
\sigma^u = (1)(2)(3 4)(5), \quad \sigma^v = (12)(3)(4)(5),
\eal
Next, the fixed point sets and their cardinalities are
\bal
\FP(\sigma^u) = \{1, 2, 5\}, \quad \fp(\sigma^u) = |\FP(\sigma^u)| = 3, \\
\FP(\sigma^v) = \{3, 4, 5\}, \quad \fp(\sigma^v) = |\FP(\sigma^v)| = 3. 
\eal
Note that $|\FP(\sigma^u) \cap \FP(\sigma^v)| = |\{5\}| = 1 \geq 1$, and there exists a coupling $(U_i, V_i)$ such that $U_i$ is uniform on $\FP(\sigma^u)$, $V_i$ is uniform on $\FP(\sigma^v)$, and 
\bal
P(U_i = V_i) = \frac{|\FP(\sigma^u) \cap \FP(\sigma^v)|}{|\FP(\sigma^v)|} = \frac{1}{3} \geq \frac{1}{5} = \frac{1}{k}.
\eal
Finally we construct $U$ (respectively, $V$) by setting the $i$th coordinate equal to $U_i$ (respectively, $V_i$) for all $1 \leq i \leq 5$. This completes the second substep. For example, suppose we obtain
\bal
U = (1, 2, 1, 2, 5), \quad V = (3, 4, 4, 3, 5).
\eal
Then by the coupling construction, the $i$th coordinates of $U$ and $V$ are equal with probability $1/3$. This ends our example. 
\end{example}

Using Theorem \ref{maintheorem}, we can also obtain an upper bound on the second largest eigenvalue of $K$. 

\begin{proposition} \label{eigenvalueUBoriginal}
The second largest eigenvalue of the transition matrix, $K$, of the Burnside process on $[k]^n$ is upper bounded by $\lambda_1 \leq 1 - \frac{1}{2k}$.
\end{proposition}

\begin{proof}
Let $H_t$ be the continuous-time chain with transition matrix $K$. Then for all $u \in [k]^n$, we have that 
\bal
2\|H_t(u,\cdot) - \pi(\cdot)\|_{TV} &= \sum_{v \in [k]^n} |H_t(u,v) - \pi(v)| = \sum_{v \in [k]^n} \left| e^{-t} \sum_{\ell = 0}^\infty \frac{t^\ell}{\ell!} [K^\ell(u,v) - \pi(v)] \right| \\ 
&\leq e^{-t} \sum_{\ell = 0}^\infty \frac{t^\ell}{\ell!} \sum_{v \in [k]^n} |K^\ell(u,v) - \pi(v)| \leq 2e^{-t} \sum_{\ell = 0}^\infty \frac{t^\ell}{\ell!} \| K_u^\ell - \pi \|_{TV} \\
&\leq 2ne^{-t} \sum_{\ell = 0}^\infty \frac{t^\ell}{\ell!} \left( 1 - \frac{1}{2k} \right)^\ell = 2ne^{-t/2k}, 
\eal
where we used Theorem \ref{maintheorem} in the third inequality. 

Combining this upper bound with Proposition \ref{spectralgapheatkernelrelation} yields
\bal
1 - \lambda_1 &= \lim_{t \to \infty} \frac{-1}{t} \log\left( 2 \max_{u \in [k]^n} \| H_t(u,\cdot) - \pi(\cdot) \|_{TV} \right) \geq \lim_{t \to \infty} \frac{-\log(2ne^{-t/2k})}{t} \\
&= \lim_{t \to \infty} \left( \frac{-\log(2n)}{t} + \frac{1}{2k}  \right) = \frac{1}{2k}, 
\eal
from which it follows that $\lambda_1 \leq 1 - \frac{1}{2k}$. 
\end{proof}

\subsection{Minorization Bound}

Consider the Burnside process on $[k]^n$ in the $1 \leq k < n$ regime. This chain has the same transition probabilities $K(u,v)$ as the $k \geq n \geq 1$ regime, but has a different stationary distribution given by $\pi(u) = \frac{(k-j_u)!}{k! \sum_{i=1}^k \stirling{n}{i}}$ where $j_u$ is the number of distinct entries in $u$. Thus the Burnside process on $[k]^n$ in this $1 \leq k < n$ regime provides a Markov chain Monte Carlo algorithm for sampling a set partition of $[n]$ of {\em at most $k$ blocks} uniformly at random. 

Using minorization, we obtain a stronger mixing time upper bound in the $1 \leq k < n$ regime than the one given in Theorem \ref{maintheorem}. 
We begin with a lemma that lower bounds the transition probabilities $K(u,v)$. 

\begin{lemma} \label{transitionlowerbound}
For $1 \leq k < n$, the transition probability of the Burnside process on $[k]^n$ satisfies
\bal
K(u,v) \geq \frac{1}{(k-1)! k^n}
\eal
for all $u,v \in [k]^n$. 
\end{lemma}

\begin{proof}
By inspection of the transition probability, $K(u,v) = \frac{(k-j)!}{(k-j_u)!} E\left[ \frac{1}{(Y + j)^n} \right]$, we have that $K(u,v)$ is minimized when $j_u = 1$ and $j_v = k$, so that there are $j = k$ distinct entries in $u$ and $v$ combined. To see this, note that the coefficient $\frac{(k-j)!}{(k-j_u)!}$ is minimized when both $(k-j)!$ is minimized and $(k-j_u)!$ is maximized, which occurs at $j_u = 1$ and $j = k$. Now let $\sigma \in S_{k-j}$ be a uniformly random permutation and recall that $Y = \fp(\sigma)$. Then $Y \leq k-j$ and by monotonicity it follows that $E\left[ \frac{1}{(Y+j)^n} \right] \geq \frac{1}{k^n}$. Therefore 
\bal
K(u,v) &= \frac{(k-j)!}{(k-j_u)!} E\left[ \frac{1}{(Y+j)^n} \right] \geq \frac{1}{(k-1)! k^n}. \qedhere
\eal
\end{proof}

We now prove an upper bound on the mixing time for the Burnside process on $[k]^n$ in the $1 \leq k < n$ regime. 

\begin{theorem} \label{minorizationbound}
For the Burnside process on $[k]^n$ in the $1 \leq k < n$ regime, the mixing time is upper bounded by
\bal
t_\mix(\epsilon) \leq \left\lceil (k-1)! \log\left( \epsilon^{-1} \right) \right\rceil.
\eal 
\end{theorem}

\begin{proof}
Let $\nu'(v) := \min_{u \in [k]^n} K(u,v)$ for all $v \in [k]^n$. Define $\nu(v) := \frac{\nu'(v)}{\nu'([k]^n)}$ and observe that $\nu$ is a probability distribution on $[k]^n$. Then
\bal
K(u,v) &\geq \min_{u \in [k]^n} K(u,v) = \nu'(v) = \nu'([k]^n) \nu(v). 
\eal
Set $c := \nu'([k]^n) = \sum_{v \in [k]^n} \min_{u \in [k]^n} K(u,v)$. It follows that
\bal
K(u,A) \geq c \nu(A)
\eal
for all $u \in [k]^n$ and all measurable subsets $A \subseteq [k]^n$. This shows that $K$ satisfies a minorization condition. 
By Proposition \ref{minorizationboundlemma},
\bal
\|K_x^t - \pi\|_{TV} \leq (1 - c)^t
\eal
for all $t$. 
Using the inequality $(1 - c)^t \leq e^{-c t}$ yields 
\bal
t_\mix(\epsilon) \leq \left\lceil c^{-1} \log\left( \epsilon^{-1} \right) \right\rceil. 
\eal
By Lemma \ref{transitionlowerbound}, we can lower bound $c$ by
\bal
c = \sum_{v \in [k]^n} \min_{u \in [k]^n} K(u,v) \geq \sum_{v \in [k]^n} \frac{1}{(k-1)! k^n} = \frac{1}{(k-1)!}.
\eal
Therefore we get the upper bound
\bal
t_\mix(\epsilon) &\leq \left\lceil (k-1)! \log\left( \epsilon^{-1} \right) \right\rceil. \qedhere
\eal 
\end{proof}

Observe that for fixed $k$, the mixing time is independent of $n$. 
Moreover, for fixed $k$, or $k$ growing slowly with $n$, Theorem \ref{minorizationbound} gives a better upper bound than Theorem \ref{maintheorem}.

\section{Geometric Bounds on the Second Largest Eigenvalue} \label{GeometricBounds}

In this section we obtain upper and lower bounds on the second largest eigenvalue of the lumped chain, $\bar{K}$. We then use these to obtain bounds on the relaxation time and mixing time.  
We start by using Poincar\'{e} and Cheeger bounds to obtain bounds on the second largest eigenvalue. 

\begin{lemma} \label{lumpedtransitionlowerbound}
The transition probability of the Burnside process on $\Pi_n$ satisfies 
\bal
\bar{K}(x,y) \geq \frac{1}{(n+1)^{n-1}}
\eal
for all $x,y \in \Pi_n$ and all $k \geq n \geq 1$. 
\end{lemma}

\begin{proof} 
The proof proceeds similarly as the proof of Lemma \ref{transitionlowerbound}. By inspection of the transition probability in Theorem \ref{setpartitiontransition}, $\bar{K}(x,y)$ is minimized when $j_x = 1$, $j_y = n$, and $k = n + 1$. 
Plugging these values into the formula for $\bar{K}$ yields
\bal
\bar{K}(x,y) &\geq \frac{n!}{(n-1)!} \cdot \frac{1}{(n+1)^n} + \frac{n!}{n!} \cdot \frac{1}{(n+1)^n} = \frac{n+1}{(n+1)^n} = \frac{1}{(n+1)^{n-1}}. \qedhere
\eal
\end{proof}

The next theorem provides bounds on the second largest eigenvalue of the Burnside process on set partitions. Let $(\bar{K}_{ij})$ be the transition matrix of $\bar{K}$ written in canonical form, as defined in Section \ref{SectionTransitionProbabilities}. 

\begin{theorem} \label{eigenvaluebounds}
For all $k \geq n \geq 2$, the second largest eigenvalue of the transition matrix, $\bar{K}$, of the Burnside process on $\Pi_n$ satisfies
\bal
1 - \frac{5 C_{n,k}}{B_n} \leq \lambda_1 \leq 1 - \frac{B_n}{(n+1)^{n-1}},
\eal
where the constant $C_{n,k}$ is given by
\bal
C_{n,k} = \sum_{i = 1}^{\lfloor B_n / 2 \rfloor} \sum_{j = \lfloor B_n / 2 \rfloor + 1}^{B_n} \bar{K}_{ij} .
\eal
\end{theorem}

\begin{proof}
For the upper bound, we use the canonical paths method. The underlying graph of $\bar{K}$ is the complete graph on $B_n$ vertices, where each vertex has a self loop, since $Q(x,y) > 0$ for all $x,y \in \Pi_n$. 
For all pairs $x,y \in \Pi_n$, define the canonical path $\gamma_{xy}$ to be the unique edge connecting $x$ and $y$. That is, the canonical paths are the {\em geodesics} of the underlying graph. Let $\Gamma$ be the set of all canonical paths. 
We compute the geometric quantity $\bar{\kappa}$ as
\bal
\bar{\kappa} = \max_{(x,y)} \frac{1}{\bar{\pi}(x) \bar{K}(x,y)} \bar{\pi}(x) \bar{\pi}(y) = \frac{1}{B_n \min_{(x,y)} \bar{K}(x,y)} \leq \frac{(n+1)^{n-1}}{B_n},
\eal
where we used Lemma \ref{lumpedtransitionlowerbound} in the last inequality. Hence by Proposition \ref{PoincareBound},
\bal
\lambda_1 \leq 1 - \frac{1}{\bar{\kappa}} \leq 1 - \frac{B_n}{(n+1)^{n-1}}. 
\eal

Next we turn to the lower bound. Let $x_1, \ldots, x_{B_n}$ denote the set partitions in $\Pi_n$ ordered by the number of blocks, as defined in Section \ref{SectionTransitionProbabilities}, and let $(\bar{K}_{ij})$ be the transition matrix of $\bar{K}$ written in canonical form. 
Define the set $S = \left\{x_1, \ldots, x_{\lfloor B_n / 2 \rfloor} \right\}$ and note that $\pi(S) \leq \frac{1}{2}$. Then
\bal
\frac{Q(S,S^c)}{\bar{\pi}(S) \bar{\pi}(S^c)} &= \frac{ \sum_{x \in S} \bar{\pi}(x) \sum_{y \in S^c} \bar{K}(x,y)}{\bar{\pi}(S) \bar{\pi}(S^c)} 
= \frac{\frac{1}{B_n} \sum_{x \in S} \sum_{y \in S^c} \bar{K}(x,y)}{ \frac{\left\lfloor \frac{B_n}{2} \right\rfloor \left\lceil \frac{B_n}{2} \right\rceil}{B_n^2} } \leq \frac{5}{B_n} \sum_{x \in S} \sum_{y \in S^c} \bar{K}(x,y) \\
&= \frac{5}{B_n} \sum_{x = x_1}^{x_{\lfloor B_n / 2 \rfloor}} \sum_{y = x_{\lfloor B_n / 2 \rfloor + 1}}^{x_{B_n}} \bar{K}(x,y) = \frac{5}{B_n} \sum_{i = 1}^{\lfloor B_n / 2 \rfloor} \sum_{j = \lfloor B_n / 2 \rfloor + 1}^{B_n} \bar{K}_{ij} = \frac{5C_{n,k}}{B_n}.
\eal 
Therefore by Proposition \ref{eigenvalueLB}, $\lambda_1 \geq 1 - \frac{Q(S,S^c)}{\pi(S) \pi(S^c)} \geq 1 - \frac{5C_{n,k}}{B_n}$.
\end{proof}

\noindent {\bf Remark.} We note that $C_{n,k}$ is a constant function of $n$ and $k$, which is increasing with $n$. We are unable to rigorously establish a useful upper bound on $C_{n,k}$, and so in practice one would need to compute $C_{n,k}$ for fixed values of $n$ and $k$. Our simulations suggest that $C_{n,k} \leq B_{n-1}$ for all $k \geq n \geq 6$. Harper \cite{Har67} showed that the quantity $\frac{B_{n+1}}{B_n}$ is of the order $\frac{n}{\log n}$, and in fact, $\frac{B_n}{B_{n-1}} \geq \frac{n}{\log n}$ for all $n \geq 4$. This would give a lower bound of $\lambda_1 \geq 1 - \frac{5 \log n}{n}$, and moreover this would show that the second largest eigenvalue converges to $1$. On the other hand our simulations also suggest that the second largest eigenvalue satisfies $\lambda_1 \leq 1 - \frac{\log n}{n}$. We leave these as open problems. \\

Finally, our bounds on the second largest eigenvalue directly translate to bounds on the mixing time of the Burnside process on $\Pi_n$.

\begin{proposition} \label{relaxationtimeUBLB}
For all $k \geq n \geq 2$, the mixing time of the Burnside process on $\Pi_n$ satisfies
\bal
\left( \frac{B_n}{5 C_{n,k}} - 1\right) \log\left( \frac{1}{2\epsilon} \right) \leq t_\mix(\epsilon) \leq \left\lceil \frac{(n+1)^{n-1}}{B_n} \log\left( \frac{B_n}{\epsilon} \right) \right\rceil
\eal 
where $C_{n,k}$ is the constant defined in Theorem \ref{eigenvaluebounds}. 
\end{proposition}

\begin{proof}
By Theorem \ref{eigenvaluebounds}, the relaxation time satisfies $\frac{B_n}{5 C_{n,k}} \leq t_\rel \leq \frac{(n+1)^{n-1}}{B_n}$.
Applying Proposition \ref{relaxationtimebounds} gives the desired bounds. 
\end{proof}

Therefore by combining Proposition \ref{relaxationtimeUBLB} with Lemma \ref{totalvariationdistancerelation}, we obtain a lower bound on the mixing time of the Burnside process on $[k]^n$. 

\begin{corollary} \label{mixingtimelowerbound}
For all $k \geq n \geq 2$, the mixing time of the Burnside process on $[k]^n$ satisfies
\bal
t_\mix(\epsilon) \geq \left( \frac{B_n}{5 C_{n,k}} - 1\right) \log\left( \frac{1}{2\epsilon} \right)
\eal 
where $C_{n,k}$ is the constant defined in Theorem \ref{eigenvaluebounds}. 
\end{corollary}



\section{Final Remarks} \label{SectionFinalRemarks} 

\subsection{}

There are other algorithms for uniformly sampling from set partitions of $[n]$ that are not based on a Markov chain. Perhaps the most well known is {\em Stam's algorithm} \cite{Sta83}. Pick an integer $N$ according to the distribution $\mu_n(k) = \frac{1}{eB_n} \frac{k^n}{k!}$, where $k \in \N$. Then drop $n$ labeled balls uniformly among $N$ urns. 
Finally, form a set partition of $[n]$ such that $i,j \in [n]$ are in the same block if and only if the balls labeled $i$ and $j$ are in the same urn. Stam showed that the set partition generated by this algorithm is a uniformly random set partition of size $n$. Moreover he showed that the number of empty urns is Poisson distributed and is independent of the generated set partition. Another algorithm for exact sampling is Arratia and DeSalvo's {\em probabilistic divide-and-conquer} (PDC) algorithm \cite{AD16}. 
The PDC algorithm divides the sample space into two parts, samples each part separately, then appropriately pieces them back together to form an exact sample from the target distribution. 

\subsection{}

As shown in Section \ref{SectionTransitionProbabilities}, the complicated formula for the transition probabilities of the Burnside process on $[k]^n$ and the lumped chain makes spectral decomposition intractable.
We have been unable to diagonalize the transition matrices $K$ or $\bar{K}$, for general values of $n$ and $k$, 
and we suspect that getting exact formulas for the characteristic polynomial, eigenvalues, and eigenvectors is a hard open problem. 

\subsection{}


Corollary \ref{mixingtimelowerbound} and our remark following Theorem \ref{eigenvaluebounds} suggests a lower bound of order $\frac{n}{\log n}$ for the mixing time of the Burnside process on $[k]^n$. We conjecture this to be true and our proof heuristic is as follows. It is well known (for example \cite{Har67}) that the expected number of blocks of a uniformly random set partition of $[n]$ is $\frac{n}{\log n}\left(1 + o(1) \right)$, where the term $o(1) \to 0$ as $n \to \infty$. Recall that the orbit of $u \in [k]^n$ is determined by coordinates with equal values, which are indexed by set partitions of $[n]$. Therefore under the stationary distribution $\pi(u)$, the expected number of distinct entries of $u$ is $\frac{n}{\log n}(1 + o(1))$. 

On the other hand, consider the Burnside process on $[k]^n$ started at $u^* = (1, 1, \ldots, 1)$ and suppose that $n$ is large. Let $U(t)$ be the state of the process at time $t$ started from $u^*$. From $u^*$, the process picks a uniformly random permutation $\sigma_1 \in S_k$, conditioned to have a fixed point at $1$. The number of fixed points of $\sigma_1$ is $\fp(\sigma_1) = 1 + X$, where $X$ is approximately Poisson distributed with mean $1$ when $n$ is large. Thus the expected number of fixed points of $\sigma_1$ is 2. Suppose the fixed points are $\FP(\sigma_1) = \{1, a_1\}$, with $a_1 \in [k] \setminus \{1\}$. From $\sigma_1$, the process jumps to the state $U(1) \in [k]^n$, which has entries independently equal to $1$ or $a_1$, with equal probability $\frac{1}{2}$. Conditioned on $\fp(\sigma_1) = 2$, a direct computation shows that the expected number of distinct entries of $U(1)$ is $2 - \frac{2}{2^n}$, which is approximately 2 when $n$ is large.

Continuing in this way, we can see that at time $m$ the expected number of distinct entries of state $U(m)$ is approximately $m+1$ when $n$ is large. Thus the process must be run at least an order of $\frac{n}{\log n}$ steps in order to have the expected number of distinct entries of $U(m)$ to reach an order of $\frac{n}{\log n}$. We leave it as an open problem to make this argument rigorous.

\subsection{}

In \cite{DZ21}, Diaconis and Zhong introduced a generalization of the Burnside process, called the {\em twisted Burnside process}. A direction of further study is to consider a twisted Burnside process on $[k]^n$ by replacing the uniform distribution on $S_k$ used in the first step with a distribution that is constant on conjugacy classes of $S_k$. 
Many examples of such distributions can be found in \cite{Dia88}. One can then obtain mixing time bounds on this twisted Burnside process.


\Address

\end{document}